\documentclass[12pt]{amsart}
\usepackage{amsmath}
\usepackage{amssymb}
\usepackage{amsfonts}
\usepackage{amsthm}
\usepackage{mathrsfs}
\usepackage[all]{xy}
\theoremstyle{plain}
\newtheorem*{theorem*}{Theorem}
\newtheorem{theorem}{Theorem}[section]
\newtheorem{lemma}[theorem]{Lemma}

\newtheorem{corollary}[theorem]{Corollary}
\newtheorem{proposition}[theorem]{Proposition}

\newtheorem{propdef}{Proposition-Definition}

\theoremstyle{definition}
\newtheorem{definition}[theorem]{Definition}
\newtheorem{remark}[theorem]{Remark}

\newtheorem{example}[theorem]{Example}

\begin{document}
\title[Equivariant coherent sheaves on the exotic nilpotent cone]{Equivariant coherent sheaves on the exotic nilpotent cone}
\author{Vinoth Nandakumar}
\maketitle

\begin{abstract}
Let $G=Sp_{2n}(\mathbb{C})$, and $\mathfrak{N}$ be Kato's exotic nilpotent cone. Following techniques used by Bezrukavnikov in \cite{bezrukavnikov} to establish a bijection between $\Lambda^+$, the dominant weights for an arbitrary simple algebraic group $H$, and $\textbf{O}$, the set of pairs consisting of a nilpotent orbit and a finite-dimensional irreducible representation of the isotropy group of the orbit, we prove an analogous statement for the exotic nilpotent cone. First we prove that dominant line bundles on the exotic Springer resolution $\widetilde{\mathfrak{N}}$ have vanishing higher cohomology, and compute their global sections using techniques of Broer. This allows to show that the direct images of these dominant line bundles constitute a quasi-exceptional set generating the category $D^b(Coh^G(\mathfrak{N}))$, and deduce that the resulting $t$-structure on $D^b(Coh^G(\mathfrak{N}))$ coincides with the perverse coherent $t$-structure. The desired result now follows from the bijection between costandard objects and simple objects in the heart of the $t$-structure on $D^b(Coh^G(\mathfrak{N}))$.
\end{abstract}
\tableofcontents
\section{Introduction}

Let $H$ be a simple algebraic group; denote by $\mathfrak{h}$ its Lie algebra and $\mathcal{N} \subset \mathfrak{h}$ its nilpotent cone. Let $\Lambda^+$ denote the set of dominant weights for $H$, and $\textbf{O}$ denote the set of pairs $(\mathcal{O}, L)$, where $\mathcal{O}$ is a $H$-orbit on $\mathcal{N}$, and $L$ is a finite-dimensional irreducible representation of the isotropy group $G^x$ of the orbit $\mathcal{O}$, where $x \in \mathcal{O}$. Using geometric methods, in \cite{bezrukavnikov}, Bezrukavnikov shows that there is a canonical bijection between $\Lambda^+$ and $\textbf{O}$ (a result that was previously conjectured by Lusztig and Vogan).

Now let $G=Sp_{2n}(\mathbb{C})$, and $\mathfrak{N}$ be Kato's exotic nilpotent cone, defined as follows:
\begin{align*} \mathfrak{N} := \mathbb{C}^{2n} \times \{ x \in \text{End}(\mathbb{C}^{2n}) \text{ } | \text{ } x \text{ nilpotent}, \langle xv,v \rangle = 0 \text{ } \forall \text{ } v \in \mathbb{C}^{2n} \} \end{align*}
The main purpose is to establish an exotic analogue of Bezrukavnikov's bijection, i.e. a bijection between $\Lambda^+$, the dominant weights for $G$, and $\mathbb{O}$, the set of pairs $(\mathcal{O}, L)$, where $\mathcal{O}$ is a $G$-orbit on $\mathfrak{N}$, and $L$ is a finite-dimensional irreducible representation of the isotropy group $G^{(v,x)}$ of the orbit $\mathcal{O}$, where $(v,x) \in \mathcal{O}$.

The exotic nilpotent cone was originally introduced by Kato in \cite{kato1} to study multi-parameter affine Hecke algebras, via the equivariant K-theory of the exotic Steinberg variety (following techiques used by Kazhdan, Lusztig and Ginzburg in the case of one-parameter affine Hecke algebras). The $G$-orbits on $\mathfrak{N}$ are proven to be in bijection with $\mathcal{Q}_n$, the set of bi-partitions of $n$ (and thus also with irreducible representations of the type $C$ Weyl group). In \cite{kato2}, Kato explicitly realizes this bijection via an exotic Springer correspondence; this correspondence is somewhat cleaner than the type $C$ Springer correspondence since there are no non-trivial local systems on $G$-orbits in $\mathfrak{N}$.

In \cite{enhanced}, Achar and Henderson make precise conjectures describing the intersection cohomology of orbit closures in the exotic nilpotent cone. The work of Achar, Henderson and Sommers in \cite{special} studies special pieces for $\mathfrak{N}$, which turn out to have the same number of $\mathbb{F}_q$ points as Lusztig's special pieces for the ordinary nilpotent cone. These results all demonstrate a strong connection between the exotic nilpotent cone and the ordinary nilpotent cone of type $C$; the present work draws another parallel between the geometry of the exotic nilpotent cone $\mathfrak{N}$ and the geometry of the ordinary nilpotent cone $\mathcal{N}$ of type $C$. 

Below we describe the main results of paper in more detail. 

\textbf{Section 2}: After recalling some of the basic properties of the exotic nilpotent cone, and some results of Achar, Henderson and Sommers on resolutions of special orbit closures, here we study the cohomology of dominant line bundles on the exotic Springer resolution $\widetilde{\mathfrak{N}}$. We closely follow the methods developed by Broer in \cite{broer1}, \cite{broer2} that prove analogous results in the case of the ordinary nilpotent cone $\mathcal{N}$. Defining $\mathcal{O}_{\widetilde{\mathcal{N}}}(\lambda) := p^* \mathcal{O}_{G/B}(\lambda)$ where $p: \widetilde{\mathfrak{N}} \rightarrow G/B$ is the projection, we first prove that $H^i(\widetilde{\mathcal{N}},\mathcal{O}_{\widetilde{\mathcal{N}}}(\lambda))=0$ using a theorem of Grauert-Riemenschneider. By using the additivity of the Euler characteristic along with Borel-Weil-Bott, we then compute the structure $H^0(\widetilde{\mathcal{N}},\mathcal{O}_{\widetilde{\mathcal{N}}}(\lambda))$ as a $G$-module.

\textbf{Section 3}: First we recall the theory of quasi-exceptional sets (following Section $2$ of \cite{bezrukavnikov}). Given a triangulated category $\mathcal{C}$, two ordered set of objects $\nabla=\{ \nabla_i | i \in I \}, \triangle=\{ \triangle_i | i \in I \}$ are said to constitute a dualizable quasi-exceptional set if they satisfy certain conditions. From a quasi-exceptional set generating a category $\mathcal{C}$, one obtains a $t$-structure on $\mathcal{C}$ whose heart is a quasi-hereditary category with simple objects also indexed by $I$. Letting $\mathcal{C} = D^b(Coh^G(\mathfrak{N}))$, define $\nabla_{\lambda} := R\pi_* \mathcal{O}_{\widetilde{\mathfrak{N}}}(\lambda)[\frac{\text{dim}(\mathfrak{N})}{2}]$ where $\pi:  \widetilde{\mathfrak{N}} \rightarrow \mathfrak{N}$ is the resolution of singularities. Then the main result is that $\nabla=\{ \nabla_{\lambda} | \lambda \in \Lambda^+ \} $ and $\triangle=\{ \nabla_{w_0 \cdot \lambda} | \lambda \in \Lambda^+ \}$ consitute a quasi-exceptional set generating $D^b(Coh^G(\mathfrak{N}))$. The results of Section $2$ are needed to prove this claim.

\textbf{Section 4}: After recalling some of the theory of perverse coherent sheaves (developed by Deligne and Bezrukavnikov), in this section we compare the $t$-structure constructed on $D^b(Coh^G(\mathfrak{N}))$ with the perverse coherent $t$-structure corresponding to the middle perversity. We first prove that $\nabla_{\lambda}$ is a perverse coherent sheaf for all $\lambda \in \Lambda$ (this essentially follows from the fact that $\pi: \widetilde{\mathfrak{N}} \rightarrow \mathfrak{N}$ is semi-small). We then deduce quasi-exceptional $t$-structure coincides with the perverse coherent $t$-structure. Since the simple perverse coherent sheaves are indexed by $\mathbb{O}$, the bijection between simple objects and co-standard objects in the heart of the $t$-structure on $D^b(Coh^G(\mathfrak{N}))$ gives the required bijection $\mathbb{O} \leftrightarrow \Lambda^+$.

We remark that there are a few differences between the proof of the bijection $\textbf{O} \leftrightarrow \Lambda^+$ for $\mathcal{N}$, and the bijection $\mathbb{O} \leftrightarrow \Lambda^+$ for $\mathfrak{N}$. While the canonical line bundle on the Springer resolution $\widetilde{\mathcal{N}}$ is trivial, the canonical line bundle on the exotic Springer resolution $\widetilde{\mathfrak{N}}$ is the anti-dominant line bundle $\mathcal{O}_{\widetilde{\mathfrak{N}}}(-\epsilon_1-\cdots-\epsilon_n)$. As a result, for the proofs in Section $3$, we use the twisted Weyl group action on $\Lambda$, $w \cdot \lambda = w(\lambda+\theta)-\theta$ with $\theta=\frac{1}{2}(\epsilon_1+\cdots+\epsilon_n)$ instead of the ordinary action. Second, in the case of the ordinary nilpotent cone $\mathcal{N}$, the proof that the image of the functor $R\pi_*: D^b(Coh^H(\widetilde{\mathcal{N}})) \rightarrow D^b(Coh^H(\mathcal{N}))$ generates $D^b(Coh^H(\mathcal{N}))$ (see Lemma $7$ in \cite{bezrukavnikov}) uses the Jacobson-Morozov resolution of orbit closures in $\mathcal{N}$. In the case of $\mathfrak{N}$, the resolutions of an arbitrary orbit closure $\overline{\mathbb{O}_{\mu, \nu}}$ are not yet known. However, resolutions of special orbit closures have been recently developed in \cite{special}, and these are sufficient to prove the corresponding result in the case of the exotic nilpotent cone. It seems possible that there is an easier proof of this result without using the results of \cite{special}.

\textbf{Acknowledgements.} I would like to thank my advisor Roman Bezrukavnikov for numerous helpful discussions and insights, and Pramod Achar who suggested to me the existence of an exotic analogue of the bijection between $\Lambda^+$ and $\textbf{O}$. I am also grateful to Anthony Henderson, David Vogan, and Gus Lehrer for helpful comments; and Dan Roozemund for help with Magma computations. I would also like to thank the University of Sydney for their hospitality, where part of this work was completed.

\section{Geometry of the exotic nilpotent cone}
\subsection{Recollections}

Here we recall the definition and basic properties of Kato's exotic nilpotent cone, following \cite{kato1} and Section $6$ of \cite{enhanced}. 

\begin{definition} Let $\mathbb{V}$ be a $2n$-dimensional $\mathbb{C}$-vector space with a symplectic form $\langle \cdot , \cdot \rangle$, $G = Sp(\mathbb{V})$ be the associated symplectic group, and $\mathfrak{g} = \mathfrak{sp}(\mathbb{V})$ be the associated symplectic Lie algebra. Define $\mathfrak{s} \subset \mathfrak{gl}(\mathbb{V})$ as below (note $\mathfrak{s} \oplus \mathfrak{sp}(\mathbb{V}) = \mathfrak{gl}(\mathbb{V})$):
$$ \mathfrak{s} = \{ x \in \text{End}(\mathbb{V}) | \langle xv,v \rangle = 0 \text{ } \forall \text{ } v \in \mathbb{V} \} $$ \end{definition}
If $\mathcal{N}$ denotes the set of nilpotent endomorphims in $\mathfrak{gl}(\mathbb{V})$, we define the \textbf{exotic nilpotent cone} to be $\mathfrak{N} = \mathbb{V} \times ( \mathfrak{s} \cap \mathcal{N})$. 

\begin{proposition} The orbits of $G$ on $\mathfrak{N}$ are in bijection with the poset $\mathcal{Q}_n$ of bipartitions of $n$. Under this correspondence, given $(\mu, \nu) \in \mathcal{Q}_n$, the orbit $\mathbb{O}_{\mu, \nu}$ consists of $(v,x) \in \mathfrak{N}$ such that the Jordan type of $x$ acting on the subspace $E^{x}v \subset \mathbb{V}$ is $\mu \cup \mu$, and the Jordan type of $x$ acting on the quotient $V/E^{x}v$ is $\nu \cup \nu$. Here $E^{x}$ denotes the centralizer of $x$ in $End(\mathbb{V})$, and for a partition $\lambda = (\lambda_1, \lambda_2, \cdots )$, $\lambda \cup \lambda$ denotes the partition $(\lambda_1, \lambda_1, \lambda_2, \lambda_2, \cdots )$. The closure ordering for $G$-orbits on $\mathfrak{N}$ corresponds to the natural ordering on $\mathcal{Q}_n$ (i.e. $(\mu, \nu) \geq (\mu', \nu')$ if $\sum_{1 \leq i \leq j} (\mu_i+\nu_i) \geq \sum_{1 \leq i \leq j} (\mu'_i + \nu'_i)$ and $\mu_{j+1} + \sum_{1 \leq i \leq j} (\mu_i+\nu_i) \geq \mu'_{j+1} + \sum_{1 \leq i \leq j} (\mu'_i + \nu'_i)$). \end{proposition}
\begin{proof} We refer the reader to Theorem $6.1$ of \cite{enhanced} and \cite{kato1} for the statement regarding orbits; see also Corollary $2.9$ in \cite{enhanced} and Theorem $1$ in \cite{mirabolic}. For the statement regarding orbit closures, see Theorem $6.3$ in \cite{enhanced}. \end{proof}
After fixing a Cartan subgroup $T \subset G$ and a Borel subgroup $B \subset G$,  let $(\mathbb{V} \oplus \mathfrak{s})^+ \subset \mathfrak{N}$ denote the sum of the strictly positive weight spaces in the $G$-module $\mathbb{V} \oplus \mathfrak{s}$; note that it is a $B$-module. 
\begin{definition} Let $\widetilde{\mathfrak{N}} = G \times_B (\mathbb{V} \oplus \mathfrak{s})^+$, and let $\pi: \widetilde{\mathfrak{N}} \rightarrow \mathfrak{N}$ be the map given by $\pi(g,(v,s))= (gv, gsg^{-1})$. Then $\pi$ is a resolution of singularities; accordingly we call $\widetilde{\mathfrak{N}}$ the exotic Springer resolution. \end{definition}

\subsection{Resolutions of special orbit closures and some subvarieties of $\widetilde{\mathfrak{N}}$}

In this section, for each orbit $\mathbb{O}_{\mu, \nu}$ (with $(\mu, \nu) \in \mathcal{Q}_n$), using resolutions of ``special" orbit closures constructed in \cite{special}, we will construct a subvariety $\widehat{\mathbb{O}_{\mu, \nu}} \subset \widetilde{\mathfrak{N}}$ with a map $\pi_{\mu,\nu}: \widehat{\mathbb{O}_{\mu, \nu}} \rightarrow \mathbb{O}_{\mu, \nu}$ whose fibres are acyclic. The need for this construction is to prove Lemma \ref{image} (an exotic analogue of Lemma $7$ from \cite{bezrukavnikov}). We start by recalling Achar-Henderson-Sommer's construction of resolutions for ``special" orbit closures from \cite{special}; for more details see Sections $2$ and $5$ of \cite{special}.

\begin{remark} We will use $C$-distinguished partitions in the following; alternatively one may use $B$-distinguished partitions (see \cite{special}) instead. \end{remark}

\begin{propdef} Let $\mathcal{Q}_n^{C} \subset \mathcal{Q}_n$ denote the subposet consisting of bi-partitions $(\mu, \nu)$ satisfying $\mu_i \geq \nu_i - 1, \nu_i \geq \mu_{i+1}-1$. If $\mathcal{P}_{2n}$ is the poset of partitions of $2n$ with the natural order, let $\mathcal{P}_{2n}^C \subset \mathcal{Q}_n$ denote the subposet consisting of partitions where each odd part occurs with even multiplicity. Define the map $\Phi^C: \mathcal{Q}_n \rightarrow \mathcal{P}_{2n}^C$ by sending the bi-partition $(\mu, \nu)$ to the partition obtained from the composition $(2\mu_1, 2\nu_1, 2\mu_2, 2\nu_2, \cdots )$ by replacing successive terms $(2s,2t)$  with $(s+t, s+t)$ if $s<t$. Define also the map $\widehat{\Phi}^C: \mathcal{P}_{2n}^C \rightarrow \mathcal{Q}_n^C$: given $\lambda \in \mathcal{P}_{2n}^C$, obtain the composition $\lambda'$ by first halving any even parts, and replacing any string of odd parts $(2k+1, \cdots , 2k+1)$ by $(k, k+1, \cdots , k, k+1)$; then let $\widehat{\Phi}^C(\lambda') = (\mu, \nu)$ where $\mu = (\lambda'_1, \lambda'_3, \cdots), \nu = (\lambda'_2, \lambda'_4, \cdots)$. Then the maps $\widehat{\Phi}^C$ and $\Phi^C|_{\mathcal{Q}_n^C}$ give an isomorphism of posets $\mathcal{Q}_n^C \simeq \mathcal{P}_{2n}^C$. For $(\mu, \nu) \in \mathcal{Q}_n$, denote $(\mu, \nu)^C = \widehat{\Phi}^C({\Phi}^C(\mu, \nu))$. \end{propdef}

\begin{definition} \label{filt} Let $(\mu, \nu) \in \mathcal{Q}_n, \lambda =\Phi^C(\mu, \nu)$. A $\lambda$-filtration (which we also refer to as a $(\mu, \nu)$-filtration) of $\mathbb{V}$ is a sequence of subspaces $(\mathbb{V}_{\geq a})_{a \in \mathbb{Z}}$, satisfying the following: \begin{itemize} \item $\mathbb{V}_{\geq a} \subseteq \mathbb{V}_{\geq a-1}$ \item $\mathbb{V}_{\geq 1-a}^{\perp} = \mathbb{V}_{\geq a}$  \item $\text{dim} \mathbb{V}_{\geq a} = \sum_{i \geq 1} \text{max}(\lceil \frac{\lambda_i -a}{2}\rceil ,0) := \lambda_a$ for $a \geq 1$ \end{itemize} For $(v,x) \in \mathfrak{N}$, say that a $\lambda$-filtration is ``$C$-adapted'' to $(v,x)$ if $v \in \mathbb{V}_{\geq 1}$ and $x(\mathbb{V}_{\geq a}) \subseteq \mathbb{V}_{\geq a+2}$ for all $a$. \end{definition}

\begin{proposition} If $(v,x) \in \mathbb{O}_{\mu, \nu}$, then there is a unique $(\mu, \nu)$-filtration $(\mathbb{V}_{\geq a})_{a \in \mathbb{Z}}$ that is $C$-adapted to $(v,x)$. \end{proposition} \begin{proof} See Theorem $5.5$ of \cite{special}. \end{proof}

\begin{definition} Fix a specific $(v,x) \in \mathbb{O}_{\mu, \nu}$, and the corresponding $(\mu, \nu)$-filtration $(\mathbb{V}_{\geq a})_{a \in \mathbb{Z}}$. Let $P \subset G$ be the parabolic subgroup stabilizing this isotropic flag. Define: \begin{align*} \mathfrak{s}_{\geq 2} = \{ x \in \mathfrak{s} | x(V_{\geq a}) \subseteq V_{\geq a+2} \} \end{align*} With this choice, it follows that $\mathbb{V}_{\geq 1} \oplus \mathfrak{s}_{\geq 2}$ is a $P$-submodule of $\mathbb{V} \oplus \mathfrak{s}$; let $\widetilde{\mathbb{O}_{(\mu, \nu)^C}}=G \times_P (\mathbb{V}_{\geq 1} \oplus \mathfrak{s}_{\geq 2})$.  \end{definition}

\begin{proposition} \label{specialresolution} The image of the natural map $G \times_P (\mathbb{V}_{\geq 1} \oplus \mathfrak{s}_{\geq 2}) \rightarrow \mathfrak{N}$ is $\overline{\mathbb{O}_{(\mu, \nu)^C}}$, and this map $\pi_{(\mu, \nu)^C}: \widetilde{\mathbb{O}_{(\mu, \nu)^C}} \rightarrow \overline{\mathbb{O}_{(\mu, \nu)^C}}$ is a resolution of singularities. Further, $\pi_{(\mu, \nu)^C}$ is an isomorphism restricted to $\pi_{(\mu, \nu)^C}^{-1}(\mathbb{O}_{\mu, \nu})$. \end{proposition} \begin{proof} See Theorem $5.7$ in \cite{special}. \end{proof}

Note that the definition above of $\widetilde{\mathbb{O}_{(\mu, \nu)^C}}$, and the definition in Section $2.1$ of $\widetilde{\mathfrak{N}}$ are not entirely canonical (they depend on choices of $(v,x)$, and a Cartan $T$, respectively). This is slightly inconvenient as we now want to relate the two varieties; the easiest way to fix the problem is via the equivalent, canonical descriptions below. 
\begin{lemma} Let $B' \subset P$ be a Borel subgroup. \begin{gather*} \widetilde{\mathfrak{N}} \simeq \{(0 \subset \mathbb{V}_1 \subset \cdots \subset \mathbb{V}_n \subset \cdots \subset \mathbb{V}_{2n}=\mathbb{V}), (v,s) \in \mathfrak{N} \text{ }|\text{ } \text{dim} \mathbb{V}_i = i, \mathbb{V}_i = \mathbb{V}_{2n-i}^{\perp}, v \in \mathbb{V}_n, s \mathbb{V}_{i+1} \subseteq \mathbb{V}_{i} \} \\ \widetilde{\mathbb{O}_{(\mu, \nu)^C}} \simeq  \{ (\mathbb{V}_{ \geq i}), (v,x) \in \mathfrak{N}\text{ } |\text{ } (\mathbb{V}_{ \geq i}) \text{ is $(\mu, \nu)$-adapted to } (v,x) \} \\ \widetilde{\mathfrak{N}} \supset G \times_{B'} (\mathbb{V}_{\geq 1} \oplus \mathfrak{s}_{\geq 2}) \simeq \{ ((0 \subset \mathbb{V}_1 \subset \cdots \subset \mathbb{V}_n \subset \cdots \subset \mathbb{V}_{2n}=\mathbb{V}), (v,s)) \in \widetilde{\mathfrak{N}} \text{ } |\text{ } \\ (\mathbb{V}_{\lambda_a}) \text{ is $(\mu, \nu)$-adapted to } (v,x) \} \end{gather*} \end{lemma} 
\begin{definition} Let $\widehat{\mathbb{O}_{\mu, \nu}} = G \times_{B'} (\mathbb{V}_{\geq 1} \oplus \mathfrak{s}_{\geq 2}) \cap \pi^{-1}(\overline{\mathbb{O}_{\mu, \nu}})$ be a subvariety of $\widetilde{\mathfrak{N}}$. \end{definition}
\begin{corollary} \label{acyc} Consider the natural map $\theta_{\mu, \nu}: \widehat{\mathbb{O}_{\mu, \nu}} \rightarrow \overline{\mathbb{O}_{\mu,\nu}}$. The fibres of $\theta_{\mu, \nu}$ over a point in $\mathbb{O}_{\mu,\nu}$ are flag varieties for the Levi subgroup of $P$, and consequently have vanishing cohomology in degrees greater than $0$. \end{corollary} \begin{proof} By construction, the fibres of $\theta_{\mu, \nu}$ are the same as the fibres of the composite map $G \times_{B'} (\mathbb{V}_{\geq 1} \oplus \mathfrak{s}_{\geq 2}) \rightarrow G \times_{P} (\mathbb{V}_{\geq 1} \oplus \mathfrak{s}_{\geq 2}) \rightarrow \overline{\mathbb{O}_{\mu,\nu}}$. The second map has singleton fibres over $\mathbb{O}_{\mu,\nu}$ by \ref{specialresolution}, and the first map has fibres $P/B' \simeq L/L \cap B'$ where $L$ is the Levi subgroup of $P$. The cohomology vanishing for the structure sheaf follows from Borel-Weil-Bott. \end{proof}

\begin{example} Let $n=6$, $\mu = (1^3), \nu = (3)$. Then replacing successive terms $(2s,2t)$ of the sequence $(2,6,2,0,2,0)$ with $(s+t,s+t)$ if $s<t$, we have $\Phi^C(\mu, \nu)=\lambda=(4,4,2,1,1) \in \mathcal{P}_{12}^C$. Then $\lambda'=(2,2,1,0,1,0)$, so $(\mu, \nu)^{C}=\widetilde{\Phi}^C(\lambda)=((2,1^2),(2))$. Using Definition \ref{filt}, $\mathbb{V}_{\geq 4}=0, \text{dim }\mathbb{V}_{\geq 3}=\text{dim }\mathbb{V}_{\geq 2} = 2, \text{dim }\mathbb{V}_{\geq 1}=5, \text{dim }\mathbb{V}_{\geq 0}=7, \text{dim }\mathbb{V}_{\geq -1}=\text{dim }\mathbb{V}_{\geq -2}=10, \mathbb{V}_{\geq -3}=\mathbb{V}$. Thus: \begin{align*} \widetilde{\mathbb{O}_{(\mu, \nu)^C}} = G \times_P (\mathbb{V}_{\geq 1} \oplus \mathfrak{s}_{\geq 2}) = \{ (&0 \subset \mathbb{V}_{\geq 3} \subset \mathbb{V}_{\geq 1} \subset \mathbb{V}_{\geq 0} \subset \mathbb{V}_{\geq -2} \subset \mathbb{V}), (v,x) | \\ &v \in \mathbb{V}_{\geq 1}, x \mathbb{V} \subseteq \mathbb{V}_{\geq -2}, x\mathbb{V}_{\geq -2} \subseteq \mathbb{V}_{\geq 1}, x\mathbb{V}_{\geq 0} \subseteq \mathbb{V}_{\geq 3}, x\mathbb{V}_{\geq 3} = 0 \} \end{align*} Proposition \ref{specialresolution} now implies that the fibres of the map $\pi_{{\mu, \nu}^C}: \widetilde{\mathbb{O}_{(\mu, \nu)^C}} \rightarrow \mathbb{O}_{(\mu, \nu)^C}$ over $\mathbb{O}_{\mu, \nu}$ are singletons. To check this, it suffices to compute the fibre of $\pi_{{\mu, \nu}^C}$ over the following point $(v,x) \in \mathbb{O}_{\mu, \nu}$: \begin{align*} v=\left(\begin{array}{c}
1 \\ 
0 \\ 
0 \\ 
0 \\ 
1 \\ 
1 \\ 
0 \\ 
0 \\ 
0 \\ 
0 \\ 
0 \\ 
0
\end{array} \right), x=\left( \begin{array}{cccccccccccc}
0 & 1 & � & � & � & � & � & � & � & � & � & � \\ 
� & 0 & 1 & � & � & � & � & � & � & � & � & � \\ 
� & � & 0 & 1 & � & � & � & � & � & � & � & � \\ 
� & � & � & 0 & � & � & � & � & � & � & � & � \\ 
� & � & � & � & 0 & � & � & � & � & � & � & � \\ 
� & � & � & � & � & 0 & � & � & � & � & � & � \\ 
� & � & � & � & � & � & 0 & � & � & � & � & � \\ 
� & � & � & � & � & � & � & 0 & � & � & � & � \\ 
� & � & � & � & � & � & � & � & 0 & 1 & � & � \\ 
� & � & � & � & � & � & � & � & � & 0 & 1 & � \\ 
� & � & � & � & � & � & � & � & � & � & 0 & 1 \\ 
� & � & � & � & � & � & � & � & � & � & � & 0
\end{array} \right) \end{align*} We have $\text{dim }(\text{Im }(x^2) \oplus \mathbb{C}v) = 5$; since $\text{Im }(x^2) \oplus \mathbb{C}v \subseteq \mathbb{V}_{\geq 1}$ it follows $\mathbb{V}_{\geq 1} = \text{Im }(x^2) \oplus \mathbb{C}v, \mathbb{V}_{\geq 0} = (\text{Im }(x^2) \oplus \mathbb{C}v)^{\perp}$. Since $\text{Im}(x^3) \subset \mathbb{V}_{\geq 3}$ and both vector spaces have dimension $2$, $V_{\geq 3} = \text{Im}(x^3), V_{\geq -2} = \text{Im }(x^3)^{\perp}$. Hence $\pi_{(\mu,\nu)^C}^{-1}(v,x)$ is a single flag, as expected. \end{example}

\subsection{Vanishing higher cohomology of dominant line bundles on $\widetilde{\mathfrak{N}}$}

\begin{definition} Denote by $p:\widetilde{\mathfrak{N}} \rightarrow G/B$ the natural projection. Let $\Lambda^+ \subset \Lambda$ denote respectively the dominant weights, and weight lattice of $G$. For $\lambda \in \Lambda^+$, let $\mathbb{C}_{\lambda}$ be the $1$-dimensional representation of $B$ where the torus $T$ acts by $\lambda$. Given a $B$-representation $V$, denote $\mathcal{L}_{G/B}(V)$ for the sheaf of sections of the vector bundle $G \times_B V$; in particular define $\mathcal{O}_{G/B}(\lambda) = \mathcal{L}_{G/B}(\mathbb{C}_{\lambda}^*)$. Following Borel-Weil, $H^0(G/B, \mathcal{O}_{G/B}(\lambda)) = V_{\lambda}$ is the finite dimensional irreducible $G$-module with highest weight $\lambda$. Denote $\mathcal{O}_{\widetilde{\mathfrak{N}}}(\lambda) = p^* \mathcal{O}_{G/B}(\lambda)$. \end{definition}

In this section we will prove the following proposition, borrowing techniques developed by Broer in \cite{broer2}. The corresponding statement for the nilpotent cone of a reductive group was originally proven in \cite{broer1} using a different technique; and reproven in \cite{broer2} in slightly more generality. 

\begin{theorem} \label{vanishing} For $\lambda \in \Lambda^+, H^i(\widetilde{\mathfrak{N}}, \mathcal{O}_{\widetilde{\mathfrak{N}}}(\lambda)) = 0$ for $i>0$. \end{theorem}

Recall the following theorem of Grauert-Riemenschneider in Kempf's version (see \cite{kempf}, Theorem 4, and also \cite{panyushev}, Theorem $3.4$). 

\begin{proposition} \label{kempf} Let $U$ be an algebraic variety, and $\omega_U$ be the canonical line bundle. If there is a proper generically finite morphism $U \rightarrow X$, where $X$ is an affine variety, then $H^{i}(U, \omega_U) = 0$ for $i>0$. \end{proposition}

\begin{lemma} Let $U = G \times_B ( (\mathbb{V} \oplus \mathfrak{s})^+ \oplus \mathbb{C}_{\lambda'})$. Then $H^{i}(U, \omega_U) = 0$ for $i>0$. \end{lemma}
\begin{proof} In accordance with \ref{kempf}, it suffices to find an affine variety $X$ with a proper, generically finite map $U \rightarrow X$. Let $X' = (\mathbb{V} \oplus \mathfrak{s})^+ \oplus V_{\lambda'}$, and define $\pi: U\rightarrow X'$ by $\pi(g,(x,y))=(gx,gy)$ where $g \in G, x \in (\mathbb{V} \oplus \mathfrak{s})^+, y \in \mathbb{C}_{\lambda'} \subset V_{\lambda'}$. The map $\pi$ is proper as we can factorize it as follows (note $G/B$ is projective): \begin{align*} G \times_B ( (\mathbb{V} \oplus \mathfrak{s})^+ \oplus \mathbb{C}_{\lambda'}) \hookrightarrow G \times_B ( (\mathbb{V} \oplus \mathfrak{s})^+ \oplus V_{\lambda'}) \simeq G/B \times  ((\mathbb{V} \oplus \mathfrak{s})^+ \oplus V_{\lambda'}) \rightarrow (\mathbb{V} \oplus \mathfrak{s})^+ \oplus V_{\lambda'} \end{align*} Let $X = \text{im}(\pi)$; we claim $X$ has the required property. Note that the map $U \rightarrow X$ generically has singleton fibres over a point $(s,0) \in X$, with $s \in \mathbb{V} \oplus \mathfrak{s}$. Since it is also proper, its fibres are generically finite, as required. \end{proof}

Denote the natural projection $p_U: G \times_B ((\mathbb{V} \oplus \mathfrak{s})^+ \oplus \mathbb{C}_{\lambda'}) \rightarrow G/B$, where $\lambda' = \lambda + \epsilon_1 + \cdots + \epsilon_n$.

\begin{lemma} \label{canonicalsheaf} We have an isomorphism $\omega_U \simeq p_U^* \mathcal{O}_{G/B}(\lambda)$. \end{lemma}
\begin{proof} Since $p_U: U \rightarrow G/B$ has fibre $(\mathbb{V} \oplus \mathfrak{s})^+ \oplus \mathbb{C}_{\lambda'}$, we have a short exact sequence of vector bundles on $U$ (noting $\Omega_{G/B} \simeq \mathcal{L}_{G/B}(\mathfrak{u})$):
\begin{align*} 0 \rightarrow p_U^* \mathcal{L}_{G/B}(\mathfrak{u}) \rightarrow \Omega_U \rightarrow p_U^* \mathcal{L}_{G/B}((\mathbb{V} \oplus \mathfrak{s})^{+} \oplus \mathbb{C}_{\lambda'})^*) \rightarrow 0 \end{align*} 
Since $\omega_U = \wedge^{top}(\Omega_U)$, it now follows that $\omega_U \simeq p_U^* \mathcal{L}_{G/B}(V)$ where: 
\begin{align*} V &= \wedge^{top} (\mathfrak{u} \oplus ((\mathbb{V} \oplus \mathfrak{s})^{+} \oplus \mathbb{C}_{\lambda'})^*) = \wedge^{top}(\mathfrak{u}) \otimes \wedge^{top}((\mathbb{V} \oplus \mathfrak{s})^{+*}) \otimes \mathbb{C}_{\lambda'}^* \end{align*}
Noting that, as $B$-modules, $\mathfrak{u}$ has weights $\epsilon_i - \epsilon_j$ for $i<j$, $\epsilon_i + \epsilon_j$, and $2 \epsilon_i$; and $(\mathbb{V} \oplus \mathfrak{s})^+$ has weights $\epsilon_i - \epsilon_j$ for $i<j$, $\epsilon_i + \epsilon_j$, and $\epsilon_i$; we can conclude that $V = \mathbb{C}_{\epsilon_1 + \cdots + \epsilon_n} \otimes \mathbb{C}_{\lambda'}^*=\mathbb{C}_{\lambda}^*$, as required. \end{proof}
\begin{proof}[Proof of Theorem] Using the above two Lemmas, we compute as follows (for $i>0$):
\begin{align*} 0 &= H^i (U, \omega_U) = H^i(G \times_B ((\mathbb{V} \oplus \mathfrak{s})^+ \oplus \mathbb{C}_{\lambda'}), p_U^* \mathcal{L}_{G/B}(\mathbb{C}_{\lambda}^*)) \\ &= H^i(G/B, \mathcal{L}_{G/B}(\mathbb{C}_{\lambda}^* \otimes S((\mathbb{V} \oplus \mathfrak{s})^{+*} \oplus \mathbb{C}_{\lambda'}^*))) \\ &= \bigoplus_{j,n \geq 0} H^i(G/B, \mathcal{L}_{G/B}(\mathbb{C}_{\lambda}^* \otimes S^j(\mathbb{V} \oplus \mathfrak{s})^{+*} \otimes S^n(\mathbb{C_{\lambda'}^*})))\\ &= \bigoplus_{j,n \geq 0} H^i(G/B, \mathcal{L}_{G/B}(\mathbb{C}_{\lambda+n\lambda'}^* \otimes S^j(\mathbb{V} \oplus \mathfrak{s})^{+*}))\\ &= \bigoplus_{n \geq 0} H^i(G \times_B (\mathbb{V} \oplus \mathfrak{s})^+, p^* \mathcal{L}_{G/B}(\mathbb{C}_{\lambda + n \lambda'}^*)) \end{align*} Taking the $n=0$ summand in the above gives the desired result. \end{proof} 
\begin{corollary} \label{canonicalvanishing} We have $H^i(\widetilde{\mathfrak{N}}, \omega_{\widetilde{\mathfrak{N}}})=0$ for $i>0$. \end{corollary}
\begin{proof} First we calculate that $\omega_{\widetilde{\mathfrak{N}}} = \mathcal{O}_{\widetilde{\mathfrak{N}}}(-\epsilon_1  - \cdots - \epsilon_n)$ . Following the method used in the proof of Lemma \ref{canonicalsheaf}, we have the short exact sequence: \begin{align*} 0 \rightarrow p^* \mathcal{L}_{G/B}(\mathfrak{u}) \rightarrow \Omega_{\widetilde{\mathfrak{N}}} \rightarrow p^* \mathcal{L}_{G/B}((\mathbb{V} \oplus \mathfrak{s})^{+*}) \rightarrow 0 \end{align*} So it follows that $\omega_{\widetilde{\mathfrak{N}}} = \wedge^{top}(\Omega_{\widetilde{\mathfrak{N}}}) = p^* \mathcal{L}_{G/B}(V)$, where $V = \wedge^{top}(\mathfrak{u}) \otimes \wedge^{top}((\mathbb{V} \oplus \mathfrak{s})^{+*}) = \mathbb{C}_{\epsilon_1 + \cdots + \epsilon_n}$ using above calculations. Thus $\omega_{\widetilde{\mathfrak{N}}} \simeq \mathcal{O}_{\widetilde{\mathfrak{N}}}(-\epsilon_1  - \cdots - \epsilon_n)$, as required. 

Now note that the proof of Theorem $2.14$ actually holds under the weaker assumption that $\lambda' = \lambda + \epsilon_1 + \cdots + \epsilon_n$ is dominant; since $\lambda = - \epsilon_1 - \cdots - \epsilon_n$ satisfies this hypothesis, the result follows. \end{proof}

\subsection{Global sections of dominant line bundles on $\widetilde{\mathfrak{N}}$}

In this section, we will use the results of the previous section to compute the global sections of dominant line bundles on $\widetilde{\mathfrak{N}}$. Given a coherent sheaf $\mathcal{E}$ on $G/B$, since $H^i(G/B, \mathcal{E})$ acquires the structure of a $G$-module, we adopt a convenient abuse of notation whereby $H^i(G/B, \mathcal{E})$ denotes the corresponding element of $K(Rep(G))$; denote $\chi_{\lambda}$ to be the image in $K(Rep(G))$ of the irreducible module with highest weight $\lambda$. We will also need the following notation. 

\begin{definition} \label{kostant} Let $p,p': \Lambda \rightarrow \mathbb{Z}$ denote the Kostant partition function in types $B$ and $C$ respectively, defined as follows: \begin{align*} \displaystyle\frac{1}{\prod_{i < j}(1-e^{\epsilon_i-\epsilon_j})(1-e^{\epsilon_i+\epsilon_j}) \prod_{i}(1-e^{2\epsilon_i})} &= \sum_{\mu \in \Lambda} p(\mu) e^{\mu} \\ \frac{1}{\prod_{i < j}(1-e^{\epsilon_i-\epsilon_j})(1-e^{\epsilon_i+\epsilon_j}) \prod_{i}(1-e^{\epsilon_i})} &= \sum_{\mu \in \Lambda} p'(\mu) e^{\mu} \end{align*} Denote $\rho \in \Lambda^+$ to be the half-sum of the positive roots. Also for $\mu \in \Lambda$, denote $\text{conv}(\mu)$ to be the intersection of $\Lambda$ with the convex hull of the set $\{w \mu \text{ } | \text{ } w \in W\}$ in $\Lambda \otimes_{\mathbb{Z}} \mathbb{R}$ (here $W$ denotes the Weyl group of $G$). Denote also $\text{conv}^0 (\lambda)$ be the complement of $\{w \lambda | w \in W\}$ in $\text{conv}(\lambda)$. \end{definition}

\begin{theorem} \label{globalsections} For $\lambda, \mu \in \Lambda^+$, \begin{align*} \text{dim Hom}_G(V_{\mu}, H^0(\widetilde{\mathfrak{N}}, \mathcal{O}_{\widetilde{\mathfrak{N}}}(\lambda))) = \sum_{w \in W} \text{sgn}(w) p'(w(\mu + \rho)-(\lambda + \rho)) \end{align*} \end{theorem} \begin{proof} Using Theorem \ref{vanishing} and the additivity of the Euler characteristic, we can compute that: \begin{align*} H^0(\widetilde{\mathfrak{N}}, \mathcal{O}_{\widetilde{\mathfrak{N}}}(\lambda)) &= \sum_{i \geq 0} (-1)^i H^i(\widetilde{\mathfrak{N}}, \mathcal{O}_{\widetilde{\mathfrak{N}}}(\lambda)) = \sum_{i \geq 0} (-1)^i H^i(G/B, \mathcal{L}_{G/B}(\mathbb{C}_{\lambda}^* \otimes S(\mathbb{V} \oplus \mathfrak{s})^{+*}) )) \\ &= \sum_{\mu \in \Lambda} p'(\mu) \sum_{i \ge 0} (-1)^i H^i(G/B, \mathcal{L}_{G/B}(\mathbb{C}_{\lambda+\mu}^*)) \end{align*} Above we have used the filtration of $S((\mathbb{V} \oplus \mathfrak{s})^{+*})$ by $1$-dimensional $B$-modules, where $\mathbb{C}_{\mu}^*$ occurs with multiplicity $p'(\mu)$. We continue using Borel-Weil-Bott Theorem, which states $\sum_{i \geq 0} H^i(G/B, \mathcal{L}_{G/B}(\mathbb{C}_{\lambda}^*)) = \text{sgn}(w) \chi_{w(\lambda + \rho) - \rho}$, where $w \in W$ is the unique Weyl group element such that $w(\lambda+\rho)-\rho \in \Lambda^+$; this implies that $\sum_{i \geq 0} (-1)^i H^i(G/B, \mathcal{L}_{G/B}(\mathbb{C}_{\lambda+\mu}^*)) =  \text{sgn}(w) \chi_{\mu}$ for $\mu \in \Lambda^+ \Leftrightarrow \mu = w^{-1}(\mu + \rho)-(\lambda + \rho)$. The result now follows: \begin{align*} H^0(\widetilde{\mathfrak{N}},\mathcal{O}_{\widetilde{\mathfrak{N}}}(\lambda))) &= \sum_{w \in W, \mu \in \Lambda^+} \text{sgn}(w) p'(w^{-1}(\mu + \rho)-(\lambda + \rho)) \chi_{\mu} \\ &= \sum_{\mu \in \Lambda^+} \chi_{\mu} \sum_{w \in W} \text{sgn}(w) p'(w(\mu + \rho)-(\lambda + \rho)) \end{align*} \end{proof} Recall also the Weyl character formulae (see e.g. \cite{humphreys}): 
\begin{proposition} For $\mu \in \Lambda^+$, let $m_{\mu}^{\lambda}$ denote the multiplicity of the weight $\lambda$ in $V_{\mu}$. Then \begin{align*} m_{\mu}^{\lambda} = \sum_{w \in W} \text{sgn}(w) p(w(\mu+\rho)-(\lambda+\rho)) \end{align*} \end{proposition}
\begin{proposition} \label{vanishing2} $\text{Hom}_G(V_{\mu}, H^0(\widetilde{\mathfrak{N}}, \mathcal{O}_{\widetilde{\mathfrak{N}}}(\lambda))) = 0$ unless $\lambda \in \text{conv}(\mu)$. \end{proposition} \begin{proof} From Definition \ref{kostant}, first note the following identity: \begin{align*} \prod_i (1+e^{\epsilon_i}) (\sum_{\mu \in \Lambda} p(\mu) e^{\mu}) &= \sum_{\mu \in \Lambda} p'(\mu) e^{\mu} \\ p'(\mu) &= \sum_{S \subseteq \{1, \cdots ,n\}} p(\mu - \sum_{i \in S} \epsilon_i) \\ \text{dim Hom}_G(V_{\mu}, H^0(\widetilde{\mathfrak{N}}, \mathcal{O}_{\widetilde{\mathfrak{N}}}(\lambda))) &= \sum_{w \in W} \text{sgn}(w) p'(w(\mu + \rho)-(\lambda + \rho)) \\ &= \sum_{w \in W, S \subseteq \{1, \cdots ,n\}} \text{sgn}(w) p(w(\mu+\rho)-(\lambda+\sum_{i \in S}\epsilon_i+\rho)) \\ &= \sum_{S \subseteq \{1, \cdots ,n\}} m_{\mu}^{\lambda+ \sum_{i \in S} \epsilon_i} \end{align*} It is well-known that $m_{\mu}^{\lambda} = 0$ unless $\lambda \in \text{conv}(\mu)$ (since $m_{\mu}^{\lambda} = 0$ unless $\lambda \preccurlyeq \mu$, and $m_{\mu}^{\lambda}=m_{\mu}^{w \lambda}$ for $w \in W$). Since $\lambda \in \Lambda^+$, it then follows that for the above quantity to be nonzero, $\lambda \in \text{conv}(\mu)$. \end{proof}

\section{A quasi-exceptional set generating $D^b(Coh^G(\mathfrak{N}))$}

\subsection{Recollections}

For the reader's convenience, in this section we give a summary of the results in Section $1$ of \cite{bezrukavnikov} regarding quasi-exceptional sets and quasi-hereditary categories that will be relevant to us in the next section. 

Let $\mathcal{C}$ be a triangulated category; for $C_1, C_2 \in \mathcal{C}$, denote $\text{Hom}^{\bullet}(C_1, C_2) = \oplus_n \text{Hom}(C_1, C_2[n])$ (and $\text{Hom}^{<0}(C_1,C_2) = \oplus_n \text{Hom}(C_1, C_2[n])$). Given two sets $\mathcal{S}_1, \mathcal{S}_2$ of objects in $\mathcal{C}$, define $\mathcal{S}_1 \star \mathcal{S}_2$ to be the set of objects $X$ for which there is a distinguished triangle $S_1 \rightarrow X \rightarrow S_2 \rightarrow S_1[1]$; the axioms of a triangulated category imply that this operation is associative. Given a set of objects $\mathcal{S}$ in $\mathcal{C}$, define $\langle \mathcal{S} \rangle = \mathcal{S} \cup \mathcal{S} \star \mathcal{S} \cup \mathcal{S} \star \mathcal{S} \star \mathcal{S} \cup \cdots$. Then $\langle \cup_i \mathcal{S}[i] \rangle$ is the smallest strictly full triangulated subcategory containing $\mathcal{S}$.

\begin{definition} Let $I$ be a totally ordered set. A dualizable quasi-exceptional set in $\mathcal{C}$ consists of a subset $\nabla=\{ \nabla_i |i \in I\}$, and its dual $\triangle = \{ \triangle_i | i \in I\}$, satisfying the following properties: \begin{itemize} \item $\text{Hom}^{\bullet}(\nabla_i, \nabla_j) = 0$ if $i < j$.  \item $\text{Hom}^{<0}(\nabla_i, \nabla_i) = 0$ and $\text{Hom}(\nabla_i, \nabla_i) = \mathbb{C}$ for all $i \in I$. \item $\text{Hom}(\triangle_j, \nabla_i )=0$ if $j>i$ \item $\triangle_i \simeq \nabla_i (\text{mod } \mathcal{C}/\mathcal{C}_{<i})$, where $\mathcal{C}_{<i}$ is the smallest strictly full triangulated subcategory containing the objects $\{ \nabla_j | j<i \}$, and $\mathcal{C}/\mathcal{C}_{<i}$ is the Verdier quotient category. \end{itemize} \end{definition}

\begin{proposition} Suppose that $\{\nabla_i | i \in I\}$ generate $\mathcal{C}$. Then $\mathcal{C}$ has a unique t-structure $(\mathcal{C}^{\geq 0}, \mathcal{C}^{\leq 0})$ such that $\nabla_i \in \mathcal{C}^{\geq 0}, \triangle_i \in \mathcal{C}^{\leq 0}$. It is given by $\mathcal{C}^{\geq 0} = \langle \{ \nabla_i[d], i \in I, d \leq 0 \} \rangle, \mathcal{C}^{\leq 0} = \langle \{ \triangle_i[d], i \in I, d \geq 0 \} \rangle $. \end{proposition} \begin{proof} See Proposition $1$ in [4]. \end{proof}

We define a quasi-hereditary category in preparation for the next result, which shows that the heart of the above t-structure is quasi-hereditary. 

\begin{definition} An abelian category $\mathcal{A}$ has simple objects $\{S_i\}$ indexed by an $I$ is quasi-hereditary if it satisfies the following properties: \begin{itemize} \item For each simple $S_i$, there is an object $A_i$ with a non-zero morphism $\alpha: A_i \rightarrow S_i$, known as its ``standard cover'', such that: \begin{itemize} \item $\text{ker}(\alpha_i) \in \mathcal{A}_{<i}$ \item $\text{Hom}(A_i, S_i) = \text{Ext}^1(A_i, S_i) = 0$ \end{itemize} \item For each simple $S_i$, there is an object $B_i$ with a non-zero morphism $\beta_i: S_i \rightarrow B$, known as its ``costandard hull", such that \begin{itemize} \item $\text{coker}(\beta_i) \in \mathcal{A}_{<i}$.  \item $\text{Hom}(S_i, B_i)=\text{Ext}^1(S_i, B_i)=1$. \end{itemize} \end{itemize} \end{definition}

Denote by $\tau_{\geq 0}, \tau_{\leq 0}$ the truncation functors associated to the above t-structure. Let $\triangle'_i = \tau_{\geq 0}(\triangle_i), \nabla'_i= \tau_{\leq 0}(\nabla_i)$ be objects in the heart of the $t$-structure. 

\begin{proposition} There exists a morphism $\theta_i: \triangle'_i \rightarrow \nabla'_i$, such that $S_i := \text{im}(\theta_i)$ is a simple object in $C^{\geq 0} \cap C^{\leq 0}$. Each simple object in $C^{\geq 0} \cap C^{\leq 0}$ is isomorphic to $S_i$ for some $i$. The map $\triangle'_i \rightarrow S_i$ is a standard cover of $S_i$, and the inclusion $S_i \rightarrow \nabla'_i$ is a costandard hull for $S_i$. Thus $C^{\geq 0} \cap C^{\leq 0}$ is a quasi-hereditary category. \end{proposition}

\subsection{Main Results}

Following the techniques in \cite{bezrukavnikov}, we may now prove analogous results about $\mathfrak{N}$ by applying the results in Section $1$. However, we will need to use the following twisted action of $W$ on $\Lambda$ instead of the usual one:

\begin{definition} Let $\theta = \frac{1}{2}(\epsilon_1 + \epsilon_2 + \cdots + \epsilon_n)$. Given $w \in W, \lambda \in \Lambda$, defined the twisted action $w \cdot \lambda = w(\lambda + \theta) - \theta$. \end{definition}

\begin{definition} Let $D^b(Coh^G(\mathfrak{N}))$ denote the derived bounded category of $G$-equivariant coherent sheaves on $\mathfrak{N}$. For $\lambda \in \Lambda$, define $\nabla_{\lambda} = R \pi_* \mathcal{O}_{\widetilde{\mathfrak{N}}}(\lambda)[d]$, where $d = \frac{\text{dim } \mathfrak{N}}{2}$. Given $S \subset \Lambda$, denote by $\mathfrak{D}_S$ to be the smallest strictly full triangulated subcategory of $D^b(Coh^G(\mathfrak{N}))$ containing the objects $\nabla_{\lambda}$ for $\lambda \in S$. \end{definition}
\begin{lemma}[Grothendieck-Serre Duality] \label{duality} The (equivariant) dualizing sheaf on $\widetilde{\mathfrak{N}}$ is isomorphic to $\omega_{\widetilde{\mathfrak{N}}}[2d]$; the functor $\mathcal{S}:D^b(Coh^G(\mathfrak{N})) \rightarrow D^b(Coh^G(\mathfrak{N}))^{op}$ of Grothendieck-Serre duality is an anti-autoequivalence satisfying $\mathcal{S}(\nabla_{\lambda}) = \nabla_{-\lambda - 2\theta}$. \end{lemma} \begin{proof} Since $H^i(\widetilde{\mathfrak{N}}, \omega_{\widetilde{\mathfrak{N}}}) = 0$ for $i>0$ using \ref{canonicalvanishing}, the equivariant dualizing sheaf on $\widetilde{\mathfrak{N}}$ is isomorphic to $\omega_{\widetilde{\mathfrak{N}}}[2d]$. Since $\omega_{\widetilde{\mathfrak{N}}} \simeq \mathcal{O}_{\widetilde{\mathfrak{N}}}(-\epsilon_1 - \cdots - \epsilon_n)$ (again from \ref{canonicalvanishing}), we have \begin{align*} R\underline{Hom}(\mathcal{O}_{\widetilde{\mathfrak{N}}}(\lambda)[d],\omega_{\widetilde{\mathfrak{N}}}[2d]) = \mathcal{O}_{\widetilde{\mathfrak{N}}}(-\lambda-\epsilon_1-\cdots-\epsilon_n)[d] \end{align*} Recalling that by definition, $\mathcal{S}(\mathcal{F}) = R\underline{Hom}(\mathcal{F},\mathbb{D}C)$, where $\mathbb{D}C$ is the equivariant dualizing sheaf; and that Grothendieck-Serre duality commutes with $R \pi_*$ since $\pi$ is proper, the result follows. \end{proof}

\begin{definition} For $\mu \in \Lambda$, let $\widetilde{conv}(\mu)$ denote the intersection of the convex hull of the set $\{w \cdot \mu | w \in W\}$ in $\Lambda \otimes_{\mathbb{Z}} \mathbb{R}$ with $\Lambda$, and denote by $\widetilde{conv}^0(\mu)$ to be the complement of $\{w \cdot \mu | w \in W \}$ in $\widetilde{conv}(\mu)$. \end{definition}

\begin{proposition} \label{mod} For $w \in W$, $\nabla_{w \cdot \lambda} \simeq \nabla_{\lambda} (\text{mod } \mathfrak{D}_{\widetilde{conv}^0 (\lambda)})$. \end{proposition} \begin{proof} Let $\alpha$ be a simple root for $G$ (i.e. $\alpha \in \{ \epsilon_1 - \epsilon_2, \cdots , \epsilon_{n-1}-\epsilon_n, 2\epsilon_n \}$); let $\widetilde{\alpha} = \alpha$ if $\alpha \neq 2\epsilon_n$, and $\widetilde{\alpha}=\epsilon_n$ if $\alpha = 2 \epsilon_n$. By induction on the length of the word $w \in W$, it suffices to prove that $\nabla_{\lambda} \simeq \nabla_{s_{\alpha} \cdot \lambda}(\text{mod } \mathfrak{D}_{\widetilde{conv}^0 (\lambda)})$, for $\lambda$ satisfying $\langle \lambda, \check{\alpha} \rangle = n > 0$. Let $P_{\alpha}$ be the minimal parabolic in $G$ corresponding to the root $\alpha$; denote by $p_{\alpha}$ the projection $G/B \rightarrow G/P_{\alpha}$. Since $G$ is simply connected, we may choose $\lambda' \in \Lambda$ such that $\langle \lambda', \check{\alpha} \rangle = n-1$. Define $\mathcal{V}_{\lambda'} = p_{\alpha, *} p_{\alpha}^* \mathcal{O}_{G/B}(\lambda')$ to be the $G$-equivariant vector bundle on $G/B$ with a filtration whose subquotients are $\mathcal{O}_{G/B}(\lambda'), \cdots , \mathcal{O}_{G/B}(\lambda' - (n-1)\alpha)$.

Let $(\mathbb{V} \oplus \mathfrak{s})^+_{\alpha}$ be the $P_{\alpha}$-submodule of $(\mathbb{V} \oplus \mathfrak{s})^+$ obtained by taking the sum of all weight spaces excluding the weight space corresponding to $\widetilde{\alpha}$. Let $\widetilde{\mathfrak{N}}_{\alpha} = G \times_B (\mathbb{V} \oplus \mathfrak{s})^+_{\alpha}$, and  $\iota_{\alpha}: \widetilde{\mathfrak{N}}_{\alpha} \hookrightarrow \widetilde{\mathfrak{N}}$ be the embedding of this divisor. Let $\pi_{\alpha}: \widetilde{\mathfrak{N}}_{\alpha} = G \times_B (\mathbb{V} \oplus \mathfrak{s})^+_{\alpha} \rightarrow G \times_{P_{\alpha}} (\mathbb{V} \oplus \mathfrak{s})^+_{\alpha}$ denote the projection, whose fibres are isomorphic to $\mathbb{P}^1$. Since $\langle \lambda - \lambda' - \alpha, \check{\alpha} \rangle = -1$, the restriction of $p^* \mathcal{V}_{\lambda'}(\lambda - \lambda' - \alpha)$ is isomorphic to the sum of several copies of $\mathcal{O}_{\mathbb{P}^1}(-1)$ when restricted to any fibre of $\pi_{\alpha}$. Thus we have: \begin{align*} \pi_{\alpha *}(\mathcal{O}_{\widetilde{\mathfrak{N}_{\alpha}}} \otimes p^*(\mathcal{V}_{\lambda'})(\lambda - \lambda' - \alpha)) = 0 \\ \pi_{*}(\mathcal{O}_{\widetilde{\mathfrak{N}_{\alpha}}} \otimes p^*(\mathcal{V}_{\lambda'})(\lambda - \lambda' - \alpha)) = 0 \end{align*} 
Under the embedding of the divisor $\iota_{\alpha}$, $\mathcal{O}(\widetilde{\mathfrak{N}}_{\alpha})=\mathcal{O}_{\widetilde{\mathfrak{N}}}(-\widetilde{\alpha})$ (see Lemma \ref{divisor} for a proof of this statement), giving us the exact sequence:
\begin{align*} 0 \rightarrow \mathcal{O}_{\widetilde{\mathfrak{N}}}(-\widetilde{\alpha}) \rightarrow  \mathcal{O}_{\widetilde{\mathfrak{N}}} \rightarrow \iota_{\alpha*} \mathcal{O}_{\widetilde{\mathfrak{N}}_{\alpha}} \rightarrow 0 \end{align*} Since $\pi_{*}(\mathcal{O}_{\widetilde{\mathfrak{N}_{\alpha}}} \otimes p^*(\mathcal{V}_{\lambda'})(\lambda - \lambda' - \alpha)) = 0$, tensoring this exact sequence with $p^* \mathcal{V}_{\lambda'}(\lambda - \lambda' - \alpha)$ and taking direct image under $\pi$, we have: \begin{align*} \pi_* p^*\mathcal{V}_{\lambda'}(\lambda-\lambda'-\alpha+\widetilde{\alpha}) \simeq \pi_* p^* \mathcal{V}_{\lambda'}(\lambda - \lambda' - \alpha) \end{align*} 

\textbf{Case 1:} Suppose $\alpha \neq 2 \epsilon_n$. Then $\alpha = \widetilde{\alpha}$, so: \begin{align*} \pi_* p^*\mathcal{V}_{\lambda'}(\lambda-\lambda') \simeq \pi_* p^* \mathcal{V}_{\lambda'}(\lambda - \lambda' - \alpha) \end{align*} Since $\mathcal{V}_{\lambda'}(\lambda - \lambda')$ has a filtration whose quotients are $\mathcal{O}_{G/B}(\lambda), \cdots , \mathcal{O}_{G/B}(\lambda - (n-1) \alpha)$ , we find that $\pi_* p^*\mathcal{V}_{\lambda'}(\lambda-\lambda')[d] \in [\nabla_{\lambda-(n-1)\alpha}] * \cdots * [\nabla_{\lambda}]$. Since $\nabla_{\lambda - (n-1) \alpha}, \cdots , \nabla_{\lambda - \alpha} \in \mathfrak{D}_{\widetilde{conv}^0(\lambda)}$, this implies $\pi_* p^*\mathcal{V}_{\lambda'}(\lambda-\lambda')[d] \simeq \nabla_{\lambda} (\text{mod } \mathfrak{D}_{\widetilde{conv}^0(\lambda)})$. Similarly, we find that $\pi_* p^*\mathcal{V}_{\lambda'}(\lambda-\lambda')[d] \in [\nabla_{\lambda-n\alpha}] * \cdots * [\nabla_{\lambda-\alpha}]$, and hence $\pi_* p^*\mathcal{V}_{\lambda'}(\lambda-\lambda'-\alpha)[d] \simeq \nabla_{\lambda-n \alpha} (\text{mod } \mathfrak{D}_{\widetilde{conv}^0(\lambda)})$. Comparing, the required identity follows (note $s_{\alpha} \cdot \lambda = \lambda - n \alpha$ in this case): $\nabla_{s_{\alpha} \cdot \lambda} \simeq \nabla_{\lambda} (\text{mod } \mathfrak{D}_{\widetilde{conv}^0(\lambda)})$. 

\textbf{Case 2:} Suppose $\alpha = 2 \epsilon_n$; then $\widetilde{\alpha} = \epsilon_n$, so: \begin{align*} \pi_* p^*\mathcal{V}_{\lambda'}(\lambda-\lambda'-\epsilon_n) \simeq \pi_* p^* \mathcal{V}_{\lambda'}(\lambda - \lambda' - 2\epsilon_n) \end{align*} Note that $s_{2 \epsilon_n} \cdot (\lambda - \epsilon_n) = \lambda - 2n \epsilon_n $. Similarly to the above, since $\pi_* p^*\mathcal{V}_{\lambda'}(\lambda-\lambda'-\epsilon_n)[d] \in [\nabla_{\lambda - (2n-1) \epsilon_n}] * \cdots * [\nabla_{\lambda - \epsilon_n}]$, we find $\pi_* p^*\mathcal{V}_{\lambda'}(\lambda-\lambda'-\epsilon_n)[d] \simeq \nabla_{\lambda - \epsilon_n} (\text{mod }\mathfrak{D}_{\widetilde{conv}^0(\lambda-\epsilon_n)})$. Since $\pi_* p^* \mathcal{V}_{\lambda'}(\lambda - \lambda' - 2\epsilon_n)[d] \in [\nabla_{\lambda - 2n \epsilon_n}] * \cdots * [\nabla_{\lambda -2 \epsilon_n}]$, we find $\pi_* p^*\mathcal{V}_{\lambda'}(\lambda-\lambda'-2\epsilon_n)[d] \simeq \nabla_{\lambda - 2n\epsilon_n} (\text{mod }\mathfrak{D}_{\widetilde{conv}^0(\lambda-\epsilon_n)})$. Comparing, we see that $\nabla_{\lambda - \epsilon_n} \simeq \nabla_{s_{2 \epsilon_n} \cdot (\lambda - \epsilon_n)} (\text{mod }\mathfrak{D}_{\widetilde{conv}^0(\lambda-\epsilon_n)})$. Replacing $\lambda - \epsilon_n$ by $\lambda$, we get the result. \end{proof} 

\begin{lemma} \label{divisor} We have $\mathcal{O}(\widetilde{\mathfrak{N}}_{\alpha})=\mathcal{O}_{\widetilde{\mathfrak{N}}}(-\widetilde{\alpha})$. \end{lemma} \begin{proof} By definition, $\widetilde{\mathfrak{N}}$ is the total space of the vector bundle $\mathcal{L}_{G/B}((\mathbb{V} \oplus \mathfrak{s})^+)$. Consider the pull-back line bundle $\mathcal{E}=p^* \mathcal{L}_{G/B}((\mathbb{V} \oplus \mathfrak{s})^+)$ on $\widetilde{\mathfrak{N}}$; denote by $\tau$ the canonical section of $\Gamma(\widetilde{\mathfrak{N}},\mathcal{E})$. The surjection of $B$-modules $(\mathbb{V} \oplus \mathfrak{s})^+ \rightarrow (\mathbb{V} \oplus \mathfrak{s})^+/(\mathbb{V} \oplus \mathfrak{s})_{\alpha}^+ \simeq \mathbb{C}_{\widetilde{\alpha}}$ gives a map $\mathcal{L}_{G/B}(\mathbb{V} \oplus \mathfrak{s})^+ \rightarrow \mathcal{O}_{G/B}(-\tilde{\alpha})$, and hence a map $\phi_{\alpha}: \mathcal{E} \twoheadrightarrow \mathcal{O}_{\widetilde{\mathfrak{N}}}(- \tilde{\alpha})$. Let $\tau_{\alpha} \in \Gamma(\widetilde{\mathfrak{N}},\mathcal{O}_{\widetilde{\mathfrak{N}}}(- \tilde{\alpha}))$ denote the image of $\tau$. Then by construction $\widetilde{\mathfrak{N}}_{\alpha}$ is the zero set of the section $\tau_{\alpha}$. It follows that $\mathcal{O}_{\widetilde{\mathfrak{N}}}(- \tilde{\alpha})=\mathcal{O}_{\widetilde{\mathfrak{N}}}(d \widetilde{\mathfrak{N}}_{\alpha})$ for some $d \in \mathbb{Z}^+$; but since $\text{Pic}(\widetilde{\mathfrak{N}})=\Lambda$ we find $d=1$, as required. 
\end{proof}

\begin{proposition} \label{hom} \begin{enumerate} \item For $\lambda \in \Lambda^+$, $\text{Hom}(\nabla_{\lambda}, \nabla_{\lambda})=\mathbb{C}$ and $\text{Hom}^{<0}(\nabla_{\lambda}, \nabla_{\lambda}) = 0$. \item For $\lambda \in \Lambda^+$, we have $\text{Hom}^{\bullet}(\nabla_{\mu}, \nabla_{\lambda}) = 0$ if $\lambda \notin \widetilde{conv}(\mu)$.\item For $\lambda, \mu \in \Lambda^+$, $\lambda \neq \mu$, we have $\text{Hom}^{\bullet}(\nabla_{w_0 \cdot \mu}, \nabla_{\lambda})=0$. \end{enumerate} \end{proposition}
\begin{proof} (1) By Theorem \ref{vanishing}, $\nabla_{\lambda}$ is concentrated in a single degree, so $\text{Hom}^{<0}(\nabla_{\lambda}, \nabla_{\lambda})=0$. Since $\mathfrak{N}$ has an open, smooth $G$-orbit, and map $\phi: \nabla_{\lambda} \rightarrow \nabla_{\lambda}$ is determined by its restriction to the open orbit, where it must be a scalar, it follows that $\text{Hom}(\nabla_{\lambda}, \nabla_{\lambda}) = \mathbb{C}$.

(2) Fix $\lambda$. Assuming that for all $\mu' \in \Lambda^+ \cap \widetilde{conv}^0 (\mu)$, $\text{Hom}^{\bullet}(\nabla_{\mu'}, \nabla_{\lambda}) = 0$; it suffices to prove that $\text{Hom}^{\bullet}(\nabla_{\mu}, \nabla_{\lambda}) = 0$. First assume $\mu \in \Lambda^+$ also.

Noting that $V_{\mu} \otimes \mathcal{O}_{G/B} \in Coh^G(G/B)$ has a filtration where the subquotient $\mathcal{O}_{G/B}(\lambda)$ occurs with multiplicity $m_{\mu}^{\lambda}$, if the weights of $V_{\mu}$ are $\nu_1, \cdots , \nu_k$ it follows that: 
\begin{align*} V_{\mu} \otimes \mathcal{O}_{\mathfrak{N}}[d] = \pi_* p^* (V_{\mu} \otimes \mathcal{O}_{G/B}) [d] \in  [\nabla_{\nu_1}^{m_{\mu}^{\nu_1}}] * [\nabla_{\nu_2}^{m_{\mu}^{\nu_2}}] * \cdots * [\nabla_{\nu_k}^{m_{\mu}^{\nu_k}}] 
\end{align*}
Given $\mathcal{A}, \mathcal{B} \in D^b(Coh^G(\mathfrak{N}))$, we view $\text{Hom}^{\bullet}(\mathcal{A}, \mathcal{B})$ as an element of $D(\text{Vect}_{\mathbb{C}})$. It follows now that:
\begin{align*} \text{Hom}^{\bullet}(V_{\mu} \otimes \mathcal{O}_{\mathfrak{N}}[d], \nabla_{\lambda}) \in [\text{Hom}^{\bullet}(\nabla_{\nu_1}^{m_{\mu}^{\nu_1}}, \nabla_{\lambda})] * [\text{Hom}^{\bullet}(\nabla_{\nu_2}^{m_{\mu}^{\nu_2}}, \nabla_{\lambda})] * \cdots * [\text{Hom}^{\bullet}(\nabla_{\nu_k}^{m_{\mu}^{\nu_k}}, \nabla_{\lambda})] \end{align*}
All $\nu_{i} \in \text{conv}(\mu)$; by the induction hypothesis, if $\nu_{i} \in \text{conv}^0(\mu) \subseteq \widetilde{conv}^0(\mu)$, then $\text{Hom}^{\bullet}(\nabla_{\nu_i}, \nabla_{\lambda})=0$; thus, for $\text{Hom}^{\bullet}(\nabla_{\nu_i}^{m_{\mu}^{\nu_i}}, \nabla_{\lambda}) \neq 0$, we require $\nu_{i} = w \mu$ for some $w \in W$ (note $m_{\mu}^{w \mu} = 1$ in this case). Either $w \mu \in \widetilde{conv}^0(\mu)$, or $w \mu = w \cdot \mu$. In the first case again we have $\text{Hom}^{\bullet}(\nabla_{\nu_i}, \nabla_{\lambda})=0$ by the induction hypothesis. In the second case, by Proposition \ref{mod}, $\nabla_{w \cdot \mu} \simeq \nabla_{\mu} (\text{mod } \mathfrak{D}_{\widetilde{conv}^0 (\mu)})$, and by the induction hypothesis, $\text{Hom}^{\bullet}(\mathcal{A}_{\mu}, \nabla_{\lambda}) = 0$ for any $\mathcal{A}_{\mu} \in \mathfrak{D}_{\widetilde{conv}^0 (\mu)}$, we may conclude $\text{Hom}^{\bullet}(\nabla_{w \cdot \mu}, \nabla_{\lambda})=\text{Hom}^{\bullet}(\nabla_{\mu}, \nabla_{\lambda})$. To summarize, $\text{Hom}^{\bullet}(\nabla_{\nu_i}^{m_{\mu}^{\nu_i}}, \nabla_{\lambda})$ is either $0$ or $\text{Hom}^{\bullet}(\nabla_{\mu}, \nabla_{\lambda})$. So re-writing the above: 
\begin{align*} \text{Hom}^{\bullet}(V_{\mu} \otimes \mathcal{O}_{\mathfrak{N}}[d], \nabla_{\lambda}) \in  \text{Hom}^{\bullet}(\nabla_{\mu}, \nabla_{\lambda}) * \cdots * \text{Hom}^{\bullet}(\nabla_{\mu}, \nabla_{\lambda}) \end{align*} However, using Theorems \ref{vanishing} and Proposition \ref{vanishing2}, note that (since $\lambda \in \widetilde{\text{conv}}(\mu) \rightarrow \lambda \in \text{conv}(\mu) $): \begin{align*} &\text{Hom}^{\bullet}(V_{\mu} \otimes \mathcal{O}_{\mathfrak{N}}[d], \nabla_{\lambda}) = \text{Hom}_G(V_{\mu}, R^{\bullet}\Gamma(\nabla_{\lambda}[-d])) = 0 \\ &\Rightarrow 0 \in \text{Hom}^{\bullet}(\nabla_{\mu}, \nabla_{\lambda}) * \cdots * \text{Hom}^{\bullet}(\nabla_{\mu}, \nabla_{\lambda}) \end{align*} Using Lemma $5$, on page $12$ of \cite{bezrukavnikov}, we now conclude that $\text{Hom}^{\bullet}(\nabla_{\mu}, \nabla_{\lambda}) = 0$. Finally, as stated above, $\text{Hom}^{\bullet}(\nabla_{w \cdot \mu}, \nabla_{\lambda}) = \text{Hom}^{\bullet}(\nabla_{\mu}, \nabla_{\lambda}) = 0$; since any weight $\mu' \in \Lambda$ is of the form $w \cdot \mu$ for some $\mu \in \Lambda^+$, this removes the initial assumption that $\mu \in \Lambda^+$ and completing the proof of the statement. 

(3) If $\lambda \notin \widetilde{conv}(\mu)$, then $\lambda \notin \widetilde{conv}(w_0 \cdot \mu)$, and the result follows from the previous part. So suppose that $\lambda \in \widetilde{conv}(\mu)$, i.e. that $\lambda + \theta \in \text{conv}(\mu+\theta)$; then $\mu \notin \widetilde{conv}(-\lambda - 2\theta)$. Recalling Lemma \ref{duality}, from the previous part we deduce the required statement (noting that $w_0 \cdot \mu = w_0 (\mu + \theta) - \theta = - \mu - 2 \theta$ since $w_0 \lambda = - \lambda$):
\begin{align*} \text{Hom}^{\bullet}(\nabla_{w_0 \cdot \mu}, \nabla_{\lambda}) = \text{Hom}^{\bullet}(\mathcal{S}(\nabla_{\lambda}), \mathcal{S}(\nabla_{w_0 \cdot \mu})) = \text{Hom}^{\bullet}(\nabla_{-\lambda-2\theta}, \nabla_{\mu}) = 0 \end{align*} \end{proof}
\begin{lemma} $D^b(\text{Coh}^G(\widetilde{\mathfrak{N}}))$ is generated as a triangulated category by the objects $\mathcal{O}_{\widetilde{\mathfrak{N}}}(\lambda)$ for $\lambda \in \Lambda$. \end{lemma} \begin{proof} It is well-known that $\text{Coh}^G (G/B) \simeq \text{Rep}(B)$, and is generated as a category by $\mathcal{O}_{G/B}(\lambda)$. Since $p: \widetilde{\mathfrak{N}} \rightarrow G/B$ gives $\widetilde{\mathfrak{N}}$ the structure of a $G$-equivariant vector bundle over $G/B$, using Lemma $6$ in \cite{bezrukavnikov} (see also the last paragraph in page $266$ of \cite{chrissginzburg}), the required result follows. \end{proof}
\begin{lemma} \label{image} The image of the functor $R\pi_*: D^b(Coh^G(\widetilde{\mathfrak{N}})) \rightarrow D^b(Coh^G(\mathfrak{N}))$ generates $D^b(Coh^G(\mathfrak{N}))$ as a triangulated category. \end{lemma} \begin{proof} Let $\mathfrak{D}$ denote the triangulated subcategory of $D^b(Coh^G(\mathfrak{N}))$ generated by the image of $R\pi_*$. It suffices to show that any $\mathcal{F} \in Coh^G(\mathfrak{N})$ lies in $\mathfrak{D}$. Given this statement, it will then also follow that $\mathcal{F}[i] \in \mathfrak{D}$; since $\{ \mathcal{F}[i] | \mathcal{F} \in Coh^G(\mathfrak{N}), i \in \mathbb{Z}\}$ generates $D^b(Coh^G(\mathfrak{N}))$, it would then follow that $\mathfrak{D} = D^b(Coh^G(\mathfrak{N}))$. 

We proceed by induction on the dimension of the support to show that an arbitrary $\mathcal{F} \in D^b(Coh^G(\mathfrak{N}))$ lies in $\mathfrak{D}$. It suffices to construct an $\tilde{\mathcal{F}} \in D^b(Coh^G(\widetilde{\mathfrak{N}}))$, and a morphism $\phi: \mathcal{F} \rightarrow R\pi_* \tilde{\mathcal{F}}$ such that the cone $\mathcal{G}$ of $\phi$ has smaller support than $\mathcal{F}$ (since $\mathcal{G} \in \mathfrak{D}$ by induction, and $R\pi_* \tilde{\mathcal{F}} \in \mathfrak{D}$, it would follow that $\mathcal{F} \in \mathfrak{D}$). Suppose that $\mathbb{O}_{\mu, \nu}$ is open in the support of $\mathcal{F}$, for some $(\mu, \nu) \in \mathcal{Q}_n$. Denote the inclusion $\iota_{\mu, \nu}: \widehat{\mathbb{O}_{\mu, \nu}} \hookrightarrow \widetilde{\mathfrak{N}}$. Let $\mathcal{F}' = (\pi \circ \iota_{\mu, \nu})^* \mathcal{F} \in Coh^G(\widehat{\mathbb{O}_{\mu, \nu}})$. We claim that $\tilde{\mathcal{F}}=\iota_{\mu,\nu *}\mathcal{F}'$ has the required property. First construct the map $\phi$ as the composition of the map $\mathcal{F} \rightarrow (\pi \circ \iota_{\mu,\nu})_* \mathcal{F}'$ (coming the adjointness of $(\pi \circ \iota_{\mu,\nu})^*$ and $(\pi \circ \iota_{\mu,\nu})_*$), with the map $(\pi \circ \iota_{\mu,\nu})_* \mathcal{F}' \rightarrow R\pi_*(\iota_{\mu,\nu *}\mathcal{F}')=R\pi_*\tilde{\mathcal{F}}$. Since the fibres of the map $\pi \circ \iota_{\mu, \nu}: \widehat{\mathbb{O}_{\mu, \nu}} \rightarrow \mathfrak{N}$ over the orbit $\mathbb{O}_{\mu, \nu}$ are acyclic (see Corollary \ref{acyc}), the projection formula gives that $\phi$ is an isomorphism when restricted to $\mathbb{O}_{\mu,\nu}$. Thus the cone of $\phi$ is supported on $\overline{\mathbb{O}_{\mu,\nu}} \setminus \mathbb{O}_{\mu,\nu}$, and has smaller dimension. \end{proof}
\begin{proposition} \label{generates} The category $D^b(\text{Coh}^G({\mathfrak{N}}))$ is generated as a triangulated category by the objects $\{\nabla_{\lambda} | \lambda \in \Lambda^+ \}$. \end{proposition} \begin{proof} By applying the previous $2$ lemmas, we find that (after a shift), $D^b(\text{Coh}^G(\mathfrak{N}))$ is generated by the objects $\{\nabla_{\lambda} | \lambda \in \Lambda \}$. The result will follow once we establish that $\mathfrak{D}_{\widetilde{conv}(\lambda) \cap \Lambda^+} = \mathfrak{D}_{\widetilde{conv}(\lambda)}$ for each $\lambda \in \Lambda^+$. Suppose that this is true for all $\lambda' \in \Lambda^+ \cap \widetilde{conv}^0 (\lambda)$, then it follows that $\mathfrak{D}_{\widetilde{conv}^0(\lambda) \cap \Lambda^+}=\mathfrak{D}_{\widetilde{conv}^0(\lambda)}$. But since $\nabla_{w \cdot \lambda} \simeq \nabla_{\lambda} (\text{mod } \mathfrak{D}_{\widetilde{conv}^0(\lambda) \cap \Lambda^+})$ by Proposition \ref{mod}, it follows that $\nabla_{w \cdot \lambda} \in \mathfrak{D}_{\widetilde{conv}(\lambda) \cap \Lambda^+}$; the result now follows. \end{proof}
\begin{definition} Denote $\triangle_{\lambda} = \nabla_{w_0 \cdot \lambda}$. \end{definition}
\begin{theorem} Fix any total order $\preceq$ on $\Lambda^+$, such that $\lambda \in \widetilde{conv}(\mu) \Rightarrow \lambda \preceq \mu$. Then the sets $\nabla = \{ \nabla_{\lambda} | \lambda \in \Lambda^+ \}$, and $\triangle = \{ \nabla_{w_0 \cdot \lambda} | \lambda \in \Lambda^+ \}$, respectively, are a quasi-exceptional set generating $D^b(\text{Coh}^G({\mathfrak{N}}))$, and a dual quasi-exceptional set. \end{theorem} \begin{proof} By definition (see section $3.1$), we must check the following:
\begin{itemize}
\item To check that $\nabla$ is quasi-exceptional, note that $\text{Hom}^{\bullet}(\nabla_{\lambda}, \nabla_{\lambda'}) = 0$ if $\lambda \preceq \lambda'$ (and are distinct), since in this case $\lambda' \notin \widetilde{conv}(\lambda)$ and we may apply Proposition \ref{hom}; this proposition also implies that $\text{Hom}^{<0}(\nabla_{\lambda}, \nabla_{\lambda})=0$ and $\text{End}(\nabla_{\lambda})=\mathbb{C}$.
\item To check that $\triangle$ is a dual quasi-exceptional set, note $\text{Hom}^{\bullet}(\triangle_{\lambda}, \nabla_{\lambda'})=0$ for $\lambda' \preceq \lambda$ using Proposition \ref{hom} (in fact, it is true provided $\lambda \neq \lambda'$); and that $\triangle_{\lambda} \simeq \nabla_{\lambda} (\text{mod } \mathfrak{D}_{\preceq \lambda})$ using Proposition \ref{mod}. 
\item Using Proposition \ref{generates}, one notes that $\nabla_{\lambda}$ generates $D^b(\text{Coh}^G(\mathfrak{N}))$. \end{itemize} \end{proof}

\begin{definition} Let $\mathfrak{D}^{q, \geq 0}, \mathfrak{D}^{q, \leq 0}$ denote the positive and negative subcategories corresponding to the t-structure on $D^b(Coh^G(\mathfrak{N}))$ given by the quasi-exceptional set $\{ \nabla_{\lambda} | \lambda \in \Lambda^+\}$. \end{definition}

\section{The perverse coherent $t$-structure}

\subsection{Recollections} For the reader's convenience, this section is a brief summary of some of the results in \cite{perversecoherent} that we will need in the following section. Let $X$ be an algebraic variety with an action of an algebraic group $G$; we will recount how to construct the ``perverse coherent" t-structure on $D^b(Coh^G(X))$. 

\begin{definition} Let $X^{top}$ denote the set of generic points of closed $G$-invariant subschemes in $X$. For $x \in X^{top}$, let $d(x)$ denote the Krull dimension of the subscheme $\overline{x}$. A perversity function $p$ is a function $X^{top} \rightarrow \mathbb{Z}$; associated to $p$ define the dual perversity function $p'$ by $p'(x) = - \text{dim}(x) - p(x)$. The function $p$ is ``monotone" if $p(x') \geq p(x)$ $\forall$ $x' \in \overline{x}$, and ``co-monotone" if $p'$ is monotone. \end{definition}

\begin{definition} Define $D^{p, \leq 0}, D^{p \geq 0} \subset D^b(Coh^G(X))$ via: \begin{align*} \mathcal{F} \in D^{p, \geq 0} \text{ iff } \forall x \in X^{top}, i_x^!(\mathcal{F}) \in D^{\geq p(x)}(\mathcal{O}_x - mod) \\ \mathcal{F} \in D^{p, \leq 0} \text{ iff } \forall x \in X^{top}, i_x^*(\mathcal{F}) \in D^{\leq p(x)}(\mathcal{O}_x - mod) \end{align*} \end{definition}

\begin{theorem} If the perversity function $p$ is monotone and co-monotone, then $(D^{p, \leq 0}, D^{p \geq 0})$ defines a t-structure on $D^b(Coh^G(X))$, which we call the perverse coherent t-structure. \end{theorem}

\begin{proposition} The irreducible objects in the heart $\mathcal{P} = D^{p, \leq 0} \cap D^{p \geq 0}$ of the perverse t-structure are parametrized by a $G$-orbit $\mathcal{O}$ on $X$, and an irreducible $G$-equivariant vector bundle $\mathcal{L}$ on $\mathcal{O}$; the corresponding object is denoted $j_{!*} \mathcal{L}[p(\mathcal{O})]$. \end{proposition}

\subsection{Comparing the quasi-exceptional and perverse coherent $t$-structures on $D^b(\text{Coh}^G(\mathfrak{N}))$}

Consider the perverse coherent $t$-structure on $D^b(\text{Coh}^G(\mathfrak{N}))$ arising from the middle perversity (i.e. for the orbit $\mathbb{O}_{\mu, \nu} \subset \mathfrak{N}, p(\mathbb{O}_{\mu, \nu}) = - \frac{\text{dim}(\mathbb{O}_{\mu, \nu})}{2}$); let $\mathfrak{D}^{p, \geq 0}, \mathfrak{D}^{p, \leq 0}$ denote the negative and positive subcategories as defined above, and let $\mathcal{P} = \mathfrak{D}^{p, \geq 0} \cap \mathfrak{D}^{p, \leq 0}$ denote the core of this $t$-structure. The goal of this section is to prove that the perverse coherent $t$-structure on $D^b(\text{Coh}^G(\mathfrak{N}))$ coincides with the $t$-structure arising from the quasi-exceptional set $\nabla$. 

\begin{proposition} $\nabla_{\lambda} \in \mathcal{P}$ for all $\lambda \in \Lambda$. \end{proposition}
\begin{proof} Using Lemma \ref{duality}, $\mathcal{S}(\nabla_{\lambda}) = \nabla_{-\lambda - \epsilon_1-\cdots-\epsilon_n}$, and the perverse coherent $t$-structure is self-dual with respect to Grothendieck-Serre duality, it suffices to prove just the first of two conditions defining the perverse coherent $t$-structure (here $\iota_{gen, \mu, \nu}: (\mathbb{O}_{\mu, \nu})_{\text{gen}} \hookrightarrow \mathfrak{N}$ denotes the inclusion of the generic point of the orbit $\mathbb{O}_{\mu, \nu}$): \begin{align*} \iota_{\text{gen}, \mu, \nu}^*(\nabla_{\lambda}) \in \mathfrak{D}^{\leq p(\mathbb{O}_{\mu, \nu})}(((\mathbb{O}_{\mu, \nu})_{\text{gen}})-mod) \\ \iota_{\text{gen}, \mu, \nu}^*(R\pi_* \mathcal{O}_{\widetilde{\mathfrak{N}}}(\lambda)) \in \mathfrak{D}^{\leq p(\mathbb{O}_{\mu, \nu})+d}(((\mathbb{O}_{\mu, \nu})_{\text{gen}})-mod) \end{align*}
Note first that $\pi: \widetilde{\mathfrak{N}} \rightarrow \mathfrak{N}$ is a semi-small resolution of singularities, i.e. $\text{dim}(\pi^{-1}(x)) \leq \frac{1}{2} \text{codim} (\mathbb{O}_{\mu, \nu}) =  p(\mathbb{O}_{\mu, \nu})+d$. Note also that from \cite{hartshorne}, Chapter $3$, Corollary $11.2$, if $f: X \rightarrow Y$ is a projective morphism of Noetherian schemes, such that $\text{dim}(f^{-1}(y)) \leq r$ for all $y \in Y$, then for $i>r$, $R^{i}f_* \mathcal{F} = 0$ for all $\mathcal{F} \in \text{Coh}(X)$. Since $\iota_{\text{gen}, \mu, \nu}^*$ is an exact functor, the required result follows. 
\end{proof}

\begin{proposition} The perverse coherent t-structure on $D^b(\text{Coh}^G(\mathfrak{N}))$ coincides with the quasi-exceptional $t$-structure corresponding to the set $\{ \nabla_{\lambda} | \lambda \in \Lambda \}$. \end{proposition}
\begin{proof} Above we have proven that $\nabla_{\lambda} \in \mathfrak{D}^{p, \geq 0}$, so it follows that $\nabla_{\lambda}[d] \in \mathfrak{D}^{p, \geq 0}$ for $d \leq 0$. Since $\mathfrak{D}^{q, \geq 0}$ is generated by the objects $\nabla_{\lambda}[d]$, it follows that $\mathfrak{D}^{q, \geq 0} \subseteq \mathfrak{D}^{p, \geq 0}$. Since the above proof also gives us $\triangle_{\lambda} \in \mathfrak{D}^{p, \leq 0}$, it follows that $\triangle_{\lambda}[d] \in \mathfrak{D}^{p, \leq 0}$ for $d \geq 0$; hence $\mathfrak{D}^{q, \leq 0} \subseteq \mathfrak{D}^{p, \leq 0}$. It now follows from the axioms of a t-structure that we have equality in the above inclusions, i.e. two t-structures coincide. \end{proof}

The irreducible objects in the heart $\mathcal{P}$ of the perverse coherent $t$-structure are parametrized by an irreducible $G$-equivariant vector bundle $\mathcal{L}$ on an orbit $\mathbb{O}_{\mu, \nu}$, with the corresponding perverse coherent sheaf being given by $j_{!*} \mathcal{L}[- \frac{\text{dim}(\mathbb{O}_{\mu, \nu})}{2}]$. If $x \in \mathbb{O}_{\mu, \nu}$, then $\mathbb{O}_{\mu, \nu} \simeq G/G^x$; and the irreducible $G$-equivariant vector bundles are given by $G \times_{G^x} V$, where $V$ is an irreducible representation of the isotropy group $G^x$. Thus:

\begin{proposition} The irreducibles in $\mathcal{P}$ are indexed by pairs $(\mathbb{O}_{\mu, \nu}, V)$ of an orbit $\mathbb{O}_{\mu, \nu}$ and an irreducible representation $V$ of $G^{x}$; with the irreducible object given by $IC_{\mathbb{O}_{\mu, \nu}, V} := j_{!*} (G \times_{G^x} V)[- \frac{\text{dim}(\mathbb{O}_{\mu, \nu})}{2}]$. The bijection between costandard objects and irreducible objects in $\mathcal{P}$ gives a bijection between $\Lambda$ and pairs $(\mathbb{O}_{\mu, \nu}, V)$ of an orbit and an irreducible representation of the isotropy group. \end{proposition}

\bibliographystyle{plain}

\end{document}